\documentclass[12pt]{amsart}

\usepackage{amsmath} 
\usepackage{amssymb} 
\usepackage{amsthm}
\usepackage{amscd}
\newtheorem{thm}{Theorem}[section]
\newtheorem{lem}[thm]{Lemma}
\newtheorem{lem-dfn}[thm]{Lemma-Definition}
\newtheorem{prop}[thm]{Proposition}
\newtheorem{cor}[thm]{Corollary}

\theoremstyle{definition}
\newtheorem{defn}[thm]{Definition}
\newtheorem{exam}[thm]{Example}
\newtheorem{ex}[thm]{Example}

\theoremstyle{remark}

\newtheorem{rem}[thm]{Remark}
\numberwithin{equation}{section}

\newcommand{\proref}[1]{Proposition~\ref{#1}}

\DeclareMathOperator{\Spec}{Spec}
\DeclareMathOperator{\spec}{Spec}

\DeclareMathOperator{\supp}{Supp}

\newcommand{\m}{\mathfrak m}
\newcommand{\frm}{\mathfrak{m}}

\newcommand{\Z}{\mathbb Z}

\newcommand{\N}{\mathbb N}

\newcommand{\bbZ}{\ensuremath{\mathbb Z}}

\newcommand{\cC}{\mathcal C}

\newcommand{\cO}{\mathcal O}

\newcommand{\cR}{\mathcal R}

\newcommand{\ol}[1]{\overline {#1}}

\newcommand{\defset}[2]{{\left\{#1\,\left| \,#2 \right. \right\}}}

\begin{document}
\title{Rees algebras and $p_g$-ideals in a two-dimensional normal local domain}
\author{Tomohiro Okuma}
\address[Tomohiro Okuma]{Department of Mathematical Sciences, 
Faculty of Science, Yamagata University,  Yamagata, 990-8560, Japan.}
\email{okuma@sci.kj.yamagata-u.ac.jp}
\author{Kei-ichi Watanabe}
\address[Kei-ichi Watanabe]{Department of Mathematics, College of Humanities and Sciences, 
Nihon University, Setagaya-ku, Tokyo, 156-8550, Japan}
\email{watanabe@math.chs.nihon-u.ac.jp}
\author{Ken-ichi Yoshida}
\address[Ken-ichi Yoshida]{Department of Mathematics, 
College of Humanities and Sciences, 
Nihon University, Setagaya-ku, Tokyo, 156-8550, Japan}
\email{yoshida@math.chs.nihon-u.ac.jp}
\thanks{This work was partially supported by JSPS Grant-in-Aid 
for Scientific Research (C) Grant Numbers, 25400050, 26400053, 26400064}
\date{\today}
\begin{abstract}
The authors \cite{OWYgood} introduced the notion of 
$p_g$-ideals for two-dimensional excellent normal 
local domain over an algebraicaly closed field 
in terms of resolution of singularities. 
In this note, we give several ring-theoretic  
characterization of $p_g$-ideals. 
For instance, 
an $\m$-primary ideal $I \subset A$ is a $p_g$-ideal 
if and only if the Rees alegbra $\mathcal{R}(I)$ is a Cohen-Macaulay normal domain.  
\end{abstract}
\maketitle

\section{Introduction}

In \cite{Li}, Lipman proved that any integrally closed $\m$-primary ideal $I$ in a 
two-dimensional rational singularity $(A,\m)$ is stable and normal, 
that is, $I^2=QI$ holds for some (every) minimal reduction $Q$ and all power of $I$ are 
integrally closed.  
This implies that the Rees algebra $\cR(I)=\oplus_{n \ge 0} I^n$ is a 
Cohen-Macaulay normal domain.
Moreover, for any two integrally closed $\m$-primary ideals $I$, $J$ in a two-dimensional
rational singularity, one can choose general elements $a \in I$ and $b \in J$ so that 
$IJ=aJ+bI$. This fact implies that the bigraded Rees algebra $\cR(I,J)$ is a 
Cohen-Macaulay ring.   
An ideal theory in a two-dimensional rational singularity is established based upon these facts.  

\par 
In \cite{OWYgood}, the authors introduced the notion of $p_g$-ideals
for two-dimensional normal local domain $(A,\m)$ having a resolution of singularities 
over an algebraically closed field. 
Let $I \subset A$ be an $\m$-primary integrally closed ideal. 
Then there exist a resolution of singularities $f \colon X \to \Spec A$ and an 
anti-nef cycle $Z$ on $X$ so that 
$I=H^0(\mathcal{O}_X(-Z))$ and $I\mathcal{O}_X$ is invertible. 
Then we can show that 
$\ell_A(H^1(\mathcal{O}_X(-Z))\le p_g(A)$,
where $p_g(A)$ is the geometric genus of $A$.  If equality holds, then $Z$ is called a
$p_g$-cycle and $I=I_Z$ is called a $p_g$-ideal. 
We can consider that the notion  of $p_g$-ideals is a natural extension of 
integrally closed $\m$-primary ideal in a two-dimensional rational singularity. 
See \cite{OWYgood} for more details. 

\par 
The main purpose of this note is to give several ring-theoretic characterization of $p_g$-ideals. 
Namely, we prove the following theorem. 

\begin{thm}[See Theorems \ref{t:normalHP}, \ref{t:pgR}] \label{Main}
Let $(A,\m)$ be a two-dimensional 
excellent normal local domain 
over an algebraically closed field. 
Let $I \subset A$ be an $\m$-primary ideal,
and let $Q$ be a minimal reduction of $I$. 
Then the following conditions are equivalent$:$
\begin{enumerate}
 \item $I$ is a $p_g$-ideal. 
 \item $I^2=QI$ and $\overline{I^n} =I^n$ 
for every $n \ge 1$, 
where $\overline{J}$ denotes the integral closure of an ideal $J$. 
 \item The Rees algebra $\cR(I)$ is a Cohen-Macaulay  normal domain. 
 \item $\overline{e}_2(I)=0$. 
\end{enumerate} 
\end{thm}

\par 
Let us explain the organization of the paper. 
In Section 2, we recall the definition and several basic properties for $p_g$-ideals. For instance, $IJ=aJ+bI$ holds true 
for any two $p_g$-ideals $I$, $J$ and general elements 
$a \in I$, $b \in J$.  
In Section 3, we give a characterization of $p_g$-ideals in terms of normal Hilbert polynomials. 
Namely, the vanishing of the second normal Hilbert coefficient of $I$ yields that the ideal is a $p_g$-ideal.  
See Theorem \ref{t:normalHP} for more general statement. 
In Section 4, we give a characterization of $p_g$-ideals in terms of Rees algebras. 
Namely, an ideal $I$ is a $p_g$-ideal if and only if 
the Rees algebra $\cR(I)$ is a Cohen-Macaulay normal domain.  
Applying this result, one can find some examples of 
$p_g$-ideals.

\medskip
\section{Basic results}
Throughout this paper, let $(A,\m)$ be an excellent 
two-dimensional normal local domain and 
$f\:X \to \spec A$ a resolution of singularities 
with exceptional divisor $E:=f^{-1}(\m)$ unless otherwise specified. 
Let $E=\bigcup_{i=1}^rE_i$ be the decomposition into  irreducible components of $E$.

\par \vspace{2mm}
First, we recall the definition of $p_g$-ideals. 
An $\m$-primary ideal $I$ is said to be {\em represented on} $X$ if the ideal sheaf $I\cO_X$ is invertible and $I=H^0(X, I\cO_X)$. 
If $I$ is represented on $X$, then there exists an 
anti-nef cycle $Z$ such that $I\cO_X=\cO_X(-Z)$;
$I$ is also said to be \textit{represented by} $Z$ and 
write $I=I_Z$. 
Note that such an ideal $I$ is integrally closed in $A$. 
\par  
We say that $\cO_X(-Z)$ {\em has no fixed component}
if $H^0(\cO_X(-Z))\ne H^0(\cO_X(-Z-E_i))$ for every $E_i\subset E$, 
i.e., the base locus of the linear system $H^0(\cO_X(-Z))$ does not contain any component of $E$.
\par 
We denote by $h^1(\cO_X(-Z))$ 
the length $\ell_A(H^1(\cO_X(-Z)))$. 
It is known that $h^1(\cO_X)$ is independent of the choice 
of the resolution of singularities. 
The invariant $p_g(A):=h^1(\cO_X)$ 
is called the \textit{geometric genus} of $A$.

\begin{thm}[{\cite[Theorem 3.1]{OWYgood}}]\label{t:lepg}
Let $Z>0$ be a cycle. 
Suppose that $\cO_X(-Z)$ has no fixed component.
Then we have the following.
\begin{enumerate}
 \item $h^1(\cO_X(-Z))\le p_g(A)$.
 \item If $h^1(\cO_X(-Z))=p_g(A)$, then $\cO_X(-Z)$ is generated $($by global sections$)$.
 \end{enumerate} 
\end{thm}

\par 
Based upon the theorem above, the authors \cite{OWYgood} introduced the notion of $p_g$-ideals. 
The definition of $p_g$-ideal is independent of the choice 
of the  resolution of singularities 
(\cite[Lemma 3.4]{OWYgood}). 

\begin{defn}[\textbf{$p_g$-ideals, $p_g$-cycles}] \label{d:pgideal}
A cycle $Z>0$ is called a {\em $p_g$-cycle} if $\cO_X(-Z)$ is generated and $h^1(X,\cO_X(-Z))=p_g(A)$.
An $\m$-primary ideal $I$ is called a {\em $p_g$-ideal} if $I$ is represented by a $p_g$-cycle on some resolution. 
\end{defn}

\par 
Assume that $p_g(A)=0$. Such a ring $A$ is called a rational singularity. 
Then every (anti-nef) cycle is a $p_g$-cycle (Lipman \cite{Li}). 
\par 
If $p_g(A) > 0$, then there exists the smallest cycle 
$C_X > 0$ on $X$ such that $h^1(C_X)=p_g(A)$. 
The cycle $C_X$ is called the \textit{cohomological cycle}
on $X$. 
Note that if $A$ is Gorenstein and the resolution of singularities $f \colon X \to \Spec A$ is minimal, then 
$C_X=Z_{K_X}$, where $K_X$ denotes the canonical divisor and $Z_{K_X}$ is the effective cycle 
such that $K_X=-Z_{K_X}$. 

We have the following characterization of $p_g$-ideals  
in terms of cohomological cycle. 

\begin{prop} [{\cite[Proposition 3.10]{OWYgood}}] \label{p:CX} 
Assume that $p_g(A)>0$.
Let $Z>0$ be a cycle such that $\cO_X(-Z)$ has no fixed component.
Then $Z$ is a $p_g$-cycle if and only if
$\cO_{C_X}(-Z)\cong \cO_{C_X}$.
\end{prop}

\par 
The following theorem plays an important role in this paper. 

\begin{thm}[\textbf{Kato's Riemann-Roch formula} \cite{kato} ]
\label{t:kato}
Let $I$ be an $\m$-primary ideal 
such that  
$I=H^0(\cO_X(-Z))$ for some anti-nef cycle $Z$ on $X$. 
Then we have 
\[
\ell_A(A/I) + h^1(\cO_X(-Z))=-\dfrac{Z^2+K_XZ}{2} + p_g(A),
\]
where $K_X$ denotes the canonical divisor of $X$. 

\end{thm}

\medskip
In what follows, let us discuss whether $Z+Z'$ is a $p_g$-cycle. 
Let $Z$, $Z'$ be anti-nef cycles on the resolution $X \to \Spec A$
such that $\cO_X(-Z)$ and $\cO_X(-Z')$ are generated. 
Take general elements $a \in I_{Z}$, $b \in I_{Z'}$ and put 
\begin{eqnarray*}
\varepsilon(Z,Z')
&:= & \ell_A(I_{Z+Z'}/a I_{Z'}+bI_{Z}) \\
& = & p_g(A)-h^1(\cO_X(-Z))-h^1(\cO_X(-Z'))+h^1(\cO_X(-Z-Z')). 
\end{eqnarray*}
Then $0 \le \varepsilon(Z,Z') \le p_g(A)$; see \cite[Proposition 2.6]{OWYgood}.

\par 
The following proposition gives an important property of $p_g$-ideals. 

\begin{prop}[\textrm{see \cite[Theorem 3.5]{OWYgood}}]\label{p:sg}
Let $Z$, $Z'$ be cycles above. 
\begin{enumerate}
\item 
If $Z$ is a $p_g$-cycle on the resolution $X$,  
then $\varepsilon(Z, Z') =0$ for any $Z'$.   
In particular, if $a \in I_{Z}$ and $b \in I_{Z'}$ are general elements, then
\[
I_{Z+Z'}=a I_{Z'}+b I_{Z}.
\]
\item Assume that $Z$ is a $p_g$-cycle. 
Then $Z'$ is a $p_g$-cycle if and only if
so is $Z+Z'$. 
\item 
If $Z+Z'$ is a $p_g$-cycle for some cycle $Z'$, 
then so is $Z$. 
\end{enumerate}
\end{prop}

\begin{proof}
(1),(2) It follows from \cite[Theorem 3.5]{OWYgood}.
\par \vspace{2mm}
(3) Let $\alpha \in H^0(\cO_X(-Z'))$ be a general element.
From the exact sequence
\begin{equation}
\label{eq:Z1Z2}
0\to \cO_X(-Z)\xrightarrow{\times \alpha} \cO_X(-Z-Z') \to \cC \to 0
\end{equation}
we obtain $h^1(\cO_X(-Z))\ge h^1(\cO_X(-Z-Z') )=p_g(A)$.
Hence $h^1(\cO_X(-Z))=p_g(A)$ by Theorem \ref{t:lepg}. 
\end{proof}

\begin{cor}[{\cite[Corollary 3.6]{OWYgood}}] \label{c:stable}
Let $I$, $J$ be $\m$-primary integrally closed ideals. 
\begin{enumerate}
\item Assume that $I$ is a $p_g$-ideal. 
For any $\m$-primary integrally closed ideal $J$ and 
general elements $a \in I$, $b \in J$, we have 
\[
IJ=aJ+bI. 
\]
\item If $I$ and $J$ are $p_g$-ideals, then $IJ$ is also a $p_g$-ideal.  
\item If $IJ$ is a $p_g$-ideal, then so are $I$ and $J$.
\end{enumerate}
\end{cor}

\par 
The following properties characterize $p_g$-ideals;
see Theorem \ref{t:pgR}.

\begin{cor} \label{c:pgprop}
Assume that $I=I_Z$ is a $p_g$-ideal. 
Then $I^n$ is integrally closed for all $n \ge 1$, $I^2=QI$, and
$I\subset Q \colon I$. 
\end{cor}
\section{The normal Hilbert polynomials}\label{s:nHP}

Throughout this section, let $(A,\m)$ be a two-dimensional 
excellent 
normal local domain over an algebraically closed field. 
It is well-known that there exist integers $\overline{e}_0(I)$, $\overline{e}_1(I)$, $\overline{e}_2(I)$ such that 
\[
\ell_A(A/\overline{I^{n+1}})=\bar e_0(I) \genfrac{(}{)}{0pt}{0}{n+2}{2}
-\bar e_1(I)\genfrac{(}{)}{0pt}{0}{n+1}{1}
+\bar e_2(I) \quad \text{for large enough}\;\; n \gg 0. 
\]
Then 
\[
P_I(n)= \bar e_0(I) \genfrac{(}{)}{0pt}{0}{n+2}{2}
-\bar e_1(I)\genfrac{(}{)}{0pt}{0}{n+1}{1}+\bar e_2(I)
\]
is called the \textit{normal Hilbert polynomial} of $I$. 

\begin{lem}\label{l:nZ}
Let $Z>0$ be a cycle such that $\cO_X(-Z)$ has no fixed component.
Then$:$ 
\begin{enumerate}
\item $h^1(\cO_X(-nZ))\ge h^1(\cO_X(-(n+1)Z))$ for $n\ge 0$.
\item Let 
\[
n_0=\min\defset{n\in \Z_{\ge 0}}{ \; 
h^1(\cO_X(-nZ))= h^1(\cO_X(-(n+1)Z))}.
\]
Then $n_0\le p_g(A)$ and 
$h^1(\cO_X(-nZ))=h^1(\cO_X(-n_0Z))$ for $n\ge n_0$.
\end{enumerate}
\end{lem}

\begin{proof}
(1) follows from the argument of \proref{p:sg}.
\par 
(2) From the exact sequence
$$
0 \to \cO_X(-nZ) \to  \cO_X(-(n+1)Z)^{\oplus 2} \to \cO_X(-(n+2)Z)\to 0,
$$
we obtain that $h^1(\cO_X(-nZ))\ge 2 h^1(\cO_X(-(n+1)Z))
- h^1(\cO_X(-(n+2)Z))$.
Thus if $h^1(\cO_X(-nZ))=h^1(\cO_X(-(n+1)Z))$ is satisfied, then 
$h^1(\cO_X(-(n+1)Z))=h^1(\cO_X(-(n+2)Z))$ holds true.
\end{proof}

\par 
By using Riemann-Roch formula, we describe the normal Hilbert-polynpmial of $I=I_Z$ in terms of the cycle $Z$.

\begin{thm} \label{t:normalHP}
Assume that $I$ is represented by a cycle $Z > 0$. 
Let $P_I(n)$ be a normal Hilbert-polynomial of $I$. 
Then 
\begin{enumerate}
\item $P_I(n)=\ell_A(A/\ol{I^{n+1}})$ for all $n \ge p_g(A)-1$. 
\vspace{1mm}
\item $\bar e_0(I)=e_0(I)=-Z^2$.  
\vspace{1mm}
\item $\bar e_1(I)=e_0(I) - \ell_A(A/I) + \big(p_g(A)-h^1(\cO_X(-Z))\big) = \dfrac{-Z^2+ZK_X}{2}$.  
\vspace{1mm}
\item $\bar e_2(I)= p_g(A)-h^1(\cO_X(-nZ))$ for all $n\ge p_g(A)$.
\end{enumerate}
\end{thm}

\begin{proof}
It follows from the Riemann-Roch formula that
\begin{align*}
\ell_A(A/\ol{I^{n+1}})&=-\frac{(n+1)^2Z^2+(n+1)ZK_X}{2}+p_g(A)-h^1(\cO_X(-(n+1)Z)) \\
&=-Z^2 \genfrac{(}{)}{0pt}{0}{n+2}{2}
-\frac{-Z^2+ZK_X}{2} \genfrac{(}{)}{0pt}{0}{n+1}{1}
+p_g(A)-h^1(\cO_X(-(n+1)Z)).
\end{align*}

Since $h^1(\cO_X(-nZ))$ is stable for $n\ge p_g(A)$ by Lemma \ref{l:nZ}, 
we obtain the required assertions. 
\end{proof}

\par 
As a corollary, we obtain a simple characterization of $p_g$-ideals in terms of 
normal Hilbert coefficients.

\begin{cor} \label{pg-Hilb}
The following conditions are equivalent$:$
\begin{enumerate}
\item $I$ is a $p_g$-ideal. 
\item $\bar e_1(I)= e_0(I) - \ell_A(A/I)$. 
\item $\bar e_2(I)=0$.
\end{enumerate}
\end{cor}

\begin{proof}
$(1) \Longrightarrow (2)$ follows from the theorem above. 
\par \vspace{2mm} 
$(2) \Longrightarrow (3):$ By assumption, $I=I_Z$ is a $p_g$-ideal. 
Hence $I_{nZ}=I^n$ is a $p_g$-ideal by Corollary \ref{c:stable}(2), and thus $\bar e_2(I)=0$ by  the theorem above. 
\par \vspace{2mm}
$(3) \Longrightarrow (1):$ The theorem above yields that 
$h^1(\cO_X(-(n+1)Z))=p_g(A)$ for $n \gg 0$ and thus 
$I_{nZ}=I^n$ is a $p_g$-ideal. 
By Corollary \ref{c:stable}(3), we obtain that $I_Z$ is also a $p_g$-ideal. 
\end{proof}

\par 
For any cycle $Z$ on $X$, we put $Z^{\bot}=\sum_{ZE_i=0}E_i$. 

\begin{prop} \label{p:coho}
Let $Z>0$ be a cycle such that $\cO_X(-Z)$ has no fixed component.
If $C$ is the cohomological cycle on $Z^{\bot}$, i.e., the smallest cycle with 
\[
 h^1(\cO_C)=\max_{D>0, D_{red}\le Z^{\bot}}h^1(\cO_D),
\]
then $ \cO_C \cong \cO_C(-n_0Z)$ and 
$h^1(\cO_C)=h^1(\cO_X(-n_0Z))=p_g(A)-\bar e_2(I_Z)$, where 
$n_0$ is an integer given by Lemma $\ref{l:nZ}$.  
\end{prop}

\begin{proof}
Let $D>0$ satisfy that $\supp (D)=Z^{\bot}$ and $DE_i<0$ for all $E_i\le Z^{\bot}$.
There exist $m,n\in \N$ such that $H^1(\cO_X(-nD-mZ))=0$ (cf. the proof of \cite[Proposition 3.10]{OWYgood}).
Then $H^1(\cO_X(-mZ))=H^1(\cO_{nD}(-mZ))$.
Since $\cO_{nD}(-mZ)\cong \cO_{nD}$, 
$h^1(\cO_X(-mZ))=h^1(\cO_C)$ for sufficiently large $m$.
\end{proof}

\begin{rem}
Assume that $\cO_X(-Z)$ is generated. 
Let $E^{(1)}, \dots ,E^{(k)}$ be the connected components of $Z^{\bot}$ and assume that each $E^{(i)}$ contracts to a normal surface singularity isomorphic to  $(A_i, \m_i)$.
Then we have $p_g(A)=\bar e_2(I_Z)+\sum_{i=1}^kp_g(A_i)$
(cf. \cite[Corollary 4.5]{o.pg-splice}). 
\end{rem}

\begin{exam} \label{ex:e-power}
Let $e \ge 2$ be an integer, and let $A=k[[x,y,z]]/(x^e+y^e+z^e)$. 
Then the Poincare series of $k[x,y,z]/(x^e+y^e+z^e)$ is 
equal to
\[
\sum_{k \ge 0} \ell_A (\frm^k/\frm^{k+1})t^k = 
\dfrac{1-t^e}{(1-t)^3} = 
\dfrac{1+t+t^2+\cdots+t^{e-1}}{(1-t)^2}. 
\] 
It follows that
\[
\ell_A(A/\frm^{n+1}) = 
\left\{
\begin{array}{ll}
e \displaystyle{{n+2 \choose 2}} 
- \dfrac{e(e-1)}{2} \displaystyle{{n+1 \choose 1}} + 
\dfrac{e(e-1)(e-2)}{6} & (n \ge e), \\[5mm]
\dfrac{(n+1)(n+2)(n+3)}{6} & (n \le e-1). 
\end{array}
\right.
\]
Hence 
\[
\left\{
\begin{array}{ccl}
\overline{e}_0(\frm)&=&e_0(\frm) = e, \\[2mm]
\overline{e}_1(\frm)&=&e_1(\frm) = \dfrac{e(e-1)}{2},\\[2mm] 
\overline{e}_2(\frm)&=&e_2(\frm) = 
\dfrac{e(e-1)(e-2)}{6}=p_g(A);\quad  
\text{see \cite[(4.11)]{tki-w}}. 
\end{array}
\right.
\]
In particular, $h^1(\mathcal{O}_X(-kZ))=0$ 
for every $k \ge e$. 
\par \vspace{2mm}
On the other hand, 
\[
\dfrac{e(e-1)}{2} =\overline{e}_1(\frm)=e_0(\frm)-\ell_A(A/\frm)+p_g(A)-
h^1(\mathcal{O}_X(-Z))
\]
yields $h^1(\mathcal{O}_X(-Z))=\dfrac{(e-1)(e-2)(e-3)}{6}$.
\par \vspace{2mm}
Furthermore, since 
$ZK=2 \cdot \overline{e}_1(\frm) - (-Z^2)=e(e-2)$ and $e_0(\frm^k)=k^2e$, we have 
\begin{eqnarray*}
h^1(\mathcal{O}_X(-kZ)) &=& e_0(\frm^k)-\ell_A(A/\frm^k)
+p_g(A)-\frac{-(kZ)^2+(kZ)K}{2} \\[2mm]
&=& k^2e - {k+2 \choose 3} + {e \choose 3} 
- \dfrac{k^2e+ke(e-2)}{2} \\[2mm]
&=& \dfrac{(e-k)(e-k-1)(e-k-2)}{6} = {e-k \choose 3} 
\end{eqnarray*}
for each $k =1,2,\ldots,e-1$. 
In particular, we get
\[
h^1(\mathcal{O}_X(-(e-3)Z))=1 \quad \text{and} \quad  
h^1(\mathcal{O}_X(-(e-2)Z))=0,
\] 
and thus $n_0=e-2$ in Lemma \ref{l:nZ}. 
\end{exam}
\medskip
\section{The Rees algebra}

Let $(A,\m)$ be a Cohen-Macaulay local ring of dimension $d$, and let $I$ be an ideal of $A$. 
Now consider three $A$-algebras, which are called blow-up algebras. 
\begin{eqnarray*}
\cR(I) &:= & A[It]=\bigoplus_{n\ge 0}I^nt^n \subset A[t]. \\
\mathcal{R}'(I) &:=& A[It,t^{-1}] = 
\bigoplus_{n \in \bbZ}I^nt^n \subset A[t,t^{-1}]. \\
G(I)&:=& \cR(I)/I\cR(I) \cong 
\mathcal{R}'(I)/t^{-1}\mathcal{R}'(I). 
\end{eqnarray*}
\par \vspace{2mm}
The algebra $\cR(I)$ (resp. $\mathcal{R}'(I)$, $G(I)$) is called the \textit{Rees algebra} 
(resp. \textit{the extended Rees algebra}, 
\textit{the associated graded ring}) of $I$. 
 
\par 
The main purpose of this section is to characterize 
$p_g$-ideals in terms of blow-up algebras.

\begin{thm}\label{t:pgR}
Let $(A,\m)$ be a two-dimensional excellent normal local domain over an  algebraically closed field, and let 
$I\subset A$ be an $\m$-primary ideal. 
Then the following conditions are equivalent$:$
\begin{enumerate}
 \item $I$ is a $p_g$-ideal in the sense of 
Definition $\ref{d:pgideal}$. 
 \item $I^2=QI$ for some minimal reduction $Q$ of $I$, 
and $\overline{I^n}=I^n$ holds true for every $n \ge 1$. 
 \item $\mathcal{R}(I)$ is a Cohen-Macaulay normal domain. 
 \item $\mathcal{R}'(I)$ is a Cohen-Macaulay normal domain with 
$a(G(I))<0$. 
\end{enumerate}
\end{thm}

\begin{proof}
$(1) \Longrightarrow (2):$ 
It follows from Corollary \ref{c:pgprop}.
\par \vspace{2mm} 
$(2) \Longrightarrow (3):$ 
Since $I^2=QI$ for some minimal reduction $Q$ of $I$, 
$\cR(I)$ is Cohen-Macaulay by Valabrega--Valla \cite{FormRing} and  Goto--Shimoda \cite{Goto-Shimoda}.
Moreover, since $A$ is normal and $\overline{I^n}=I^n$ for every $n \ge 1$, 
$\cR(I)$ is a normal domain.

\par \vspace{2mm} 
$(4) \Longleftrightarrow (3) \Longrightarrow (2)$ follows from  
Goto--Shimoda \cite{Goto-Shimoda} and Herzog et.al \cite[Proposition 2.1.2]{HSV}.  
\par \vspace{2mm} 
$(2) \Longrightarrow (1):$ 
Assume that $I^n$ is integrally closed for $n\ge 1$ and that $I^2=QI$ for  a minimal reduction $Q$ of $I$.
Suppose that $I$ is represented by a cycle $Z$ on $X$.
Consider the following exact sequence given by general elements of $I=I_Z$ and $I_{nZ}$ (see \cite[(2.3)]{OWYgood}).
\[
0 \to \cO_X \to  \cO_X(-Z) \oplus \cO_X(-nZ)\to \cO_X(-(n+1)Z)\to 0.
\]
Since $QI^n=I^{n+1}=\ol{{I^{n+1}}}$, we obtain that $\varepsilon(Z,nZ)=0$ for $n\ge 1$. 
Therefore, $p_g(A)=h^1(\cO_X(-Z))$ because 
$h^1(\cO_X(-nZ))$ is stable for $n\gg 0$.
\end{proof}

\par
The following two examples are known.  

\begin{exam}[\textrm{cf. Lipman \cite{Li}} ] \label{rational-pg}
Let $A$ be a two-dimensional rational singularity. 
Then any integrally closed $\m$-primary ideal $I$ is a 
$p_g$-ideal and $\mathcal{R}(I)$ is a Cohen-Macaulay normal domain. 
\end{exam}

\begin{exam}
Let $A$ be a complete Gorenstein local ring with 
$p_g(A) > 0$. 
If $\m$ is a $p_g$-ideal of $A$, then $\m$ is stable, that is, 
$\m^2=Q\m$ for some minimal reduction $Q$ of $\m$. 
Since $A$ is Gorenstein, we obtain that $A$ is a hypersurface of degree $2$. 
So we may assume that $A=K[[x,y,z]]/(f)$, where 
$f=x^2+g(y,z)$. As $A$ is not rational, $g(y,z) \in (y,z)^3$. 
Moreover, since $R(\m)$ is normal, we have 
$g(y,z) \notin (y,z)^4$. 
\par 
Conversely, if $A = K[[x,y,z]]/(x^2+g(y,z))$, where 
$g(y,z) \in (y,z)^3 \setminus (y,z)^4$, then 
for every $n$,  $\m^n=  (y,z)^n + x (y,z)^{n-1}$ and is integrally closed.
Then,  since $\m$ is stable and $\m^n$ is integrally closed for every $n\ge 1$, 
$\m$ is a $p_g$-ideal.
\end{exam}

\par 
The next example gives a hypersurface local ring $A$ 
whose maximal ideal is a $p_g$-ideal and 
$p_g(A)=p$ for a given integer $p \ge 1$. 
\begin{ex} \label{ex-pg}
Let $p \ge 1$ be an integer, and let $k$ be an algebraically closed field. 
Let $B=k[x,y,z]/(x^2+y^3+z^{6p+1})$. 
If we put $\deg x=3(6p+1)$, $\deg y=2(6p+1)$ and 
$\deg z=6$, then $A$ can be regarded as a quasi-homogeneous $k$-algebra with $a(A)=6p-5$. 
In particular, 
\[
p_g(B)=\sum_{i=0}^{6p-5} \dim_k B_i = p; \quad 
\text{(cf. \cite{p.qh, KeiWat-D})}.
\]
Moreover, if we put $X=xt$, $Y=yt$, $Z=zt$ and $U=t^{-1}$,
then the extended Rees algebra of $\m=(x,y,z)$ is 
\[
\cR'(\m) \cong k[X,Y,Z,U]/(F),
\]
where $F=X^2+Y^3U+Z^{6p+1}U^{6p-1}$.
Since the Jacobian ideal is 
\[
\left(\frac{\partial F}{\partial X},
\frac{\partial F}{\partial Y},
\frac{\partial F}{\partial Z},
\frac{\partial F}{\partial U},F
\right)=
(X,Y^2U,Z^{6p}U^{6p-1},Y^3+(6p-1)Z^{6p+1}U^{6p-2}), 
\]
one can check $(R_1)$-condition of $\cR'(\m)$. 
Thus $\cR'(\m)$ is a normal domain because it is Cohen-Macaulay. 
\par
Now let us put $A=B_{(x,y,z)}$ and $\m=(x,y,z)A$. 
Then we can conclude that $A$ is a two-dimensional normal 
hypersurface with $p_g(A)=p$ and that 
$\m$ is a $p_g$-ideal by applying the theorem above.

\par \vspace{2mm}
Similarly, if we consider $I_k=(x,y,z^k)A$ and $Q_k=(y,z^k)$ for 
$k=2,3,\ldots,3p$, then $I_k^2=Q_kI_k$ and $\cR'(I_k)$
is a normal domain. 
Hence $I_k$ is a $p_g$-ideal. 
\end{ex}

\par 
The next example gives a hypersurface local ring $A$ 
whose maximal ideal is not  a $p_g$-ideal and 
$p_g(A)=p$ for a given integer $p \ge 1$.

\begin{ex} \label{ex-nonpg}
Let $p\ge 1$ be an integer. 
Let $A=k[x,y,z]_{(x,y,z)}/(x^2+y^4+z^{4p+1})$. 
Then $A$ is a two-dimensional normal hypersurface with 
$p_g(A)=p$. 
Then $\m=(x,y,z)$ is \textit{not} a $p_g$-ideal and 
$I_k =(x,y,z^k)$ is a $p_g$-ideal for every $k=2,3, \ldots,2p$ because $\cR'(I_k)$ is normal but $\cR'(\m)$ is not. 
\par 
Furthermore, $\m^k$ is not a $p_g$-ideal for every $k \ge 1$ by Corollary \ref{c:stable}. 
\end{ex}

It is not so difficult to extend our result to the case of 
bigraded Rees algebras. 
Let $I,J \subset A$ be ideals. 
Then 
\[
\cR(I,J):=A[It_1,Jt_2] =\bigoplus_{n=1}^{\infty}\bigoplus_{m=1}^{\infty} I^{m}J^nt_1^{m}t_2^n \subset A[t_1,t_2]
\]
is called the \textit{multi-Rees algebra} of $I$ and $J$.

\begin{cor} \label{c:multi}
Let $(A,\m)$ be a two-dimensional excellent normal local domain 
over an algebraically closed field, 
and let $I$, $J$ be $\m$-primary ideals. 
Then the following conditions are equivalent$:$
\begin{enumerate}
 \item $I$ and $J$ are $p_g$-ideals. 
 \item $\cR(I,J)$ is a Cohen-Macaulay normal domain. 
 \item $I$, $J$ are integrally closed and 
$\mathcal{R}(IJ)$ is a Cohen-Macaulay normal domain. 
\end{enumerate} 
\end{cor}

\begin{proof}
$(1) \Longrightarrow (2):$ 
Since $I$ and $J$ are $p_g$-ideals, $\cR(I)$ and $\cR(J)$ 
are Cohen-Macaulay and $IJ=aJ+bI$ for some joint reduction $(a,b)$ of $(I,J)$. 
Hence $\cR(I,J)$ is Cohen-Macaulay by \cite[Corollary 3.5]{Hy} 
(see also e.g. \cite{HHRT}, \cite{Ver1}, \cite{Ver2}).
Since $S=\cR(I)$ is a normal domain and $JI^k$ is integrally closed 
for every $k \ge 1$,  $\cR(I,J)$ is normal. 
\par \vspace{2mm}
$(2) \Longrightarrow (1):$ 
Since $\cR(I,J)$ is Cohen-Macaulay, $\cR(I)$ and $\cR(J)$ are Cohen-Macaulay by \cite[Corollary 3.5]{Hy}. 
Since $\cR(I)$ and $\cR(J)$ are pure subrings of $\cR(I,J)$, 
they are normal domains. 
Hence $I$ and $J$ are $p_g$-ideals by Theorem \ref{t:pgR}.
\par \vspace{2mm}
$(1) \Longleftrightarrow (3):$
It follows from Theorem \ref{t:pgR}and Corollary \ref{c:stable}. 
\end{proof}

\begin{rem}
By a similar argument as in the proof of 
$(1) \Longrightarrow (2)$, we can obtain that 
the multi-Rees algebra $\mathcal{R}(I_1,\ldots,I_r)$ 
is a Cohen-Macaulay normal domain for every 
$p_g$-ideals $I_1,\ldots,I_r$. 
\end{rem}

\begin{rem}
Assume that $A$ is a rational singularity. 
Let $I$ and $J$ be $\m$-primary integrally closed ideals of $A$. 
Then $I$ and $J$ are $p_g$-ideals and thus 
$\cR(I)$, $\cR(J)$ and $\cR(I,J)$ are Cohen-Macaulay normal domains. 
In fact, S. Goto, N. Taniguchi and the third author \cite{GTY} prove that 
$\cR(I)$ and $\cR(J)$ are almost Gorenstein. Moreover, Verma \cite{Ver2} 
proved that they admit minimal multiplicities. 
\end{rem}


\providecommand{\bysame}{\leavevmode\hbox to3em{\hrulefill}\thinspace}
\providecommand{\MR}{\relax\ifhmode\unskip\space\fi MR }
\providecommand{\MRhref}[2]{%
  \href{http://www.ams.org/mathscinet-getitem?mr=#1}{#2}
}
\providecommand{\href}[2]{#2}

\end{document}